\documentclass[11pt,reqn,reqno]{amsart}%
\usepackage{amsmath,amsfonts,amssymb,wasysym,amsmath}
\usepackage[latin1]{inputenc}
\usepackage{color}
\usepackage{geometry}
\usepackage{amsmath}
\usepackage{amsfonts}
\usepackage{amssymb}
\usepackage{graphicx}%
\providecommand{\U}[1]{\protect\rule{.1in}{.1in}}
\newcommand{\R}{{\mathbb R}}

\newcommand{\be}[1]{\begin{equation}\label{#1}}
\newcommand{\ee}{\end{equation}}

\newtheorem{theorem}{Theorem}[section]

\newtheorem{proposition}{Proposition}[section]

\newtheorem{definition}{Definition}[section]
\newtheorem{example}{Example}[section]
\newtheorem{remark}{Remark}[section]
\makeatother
\begin{document}
\title[Comparison and regularity results for the fractional Laplacian]{Comparison and regularity results for the fractional Laplacian via
symmetrization methods}
\author{Giuseppina Di Blasio}
\address[GIUSEPPINA DI BLASIO]{Dipartimento di Matematica, Seconda Universit\`a degli Studi di Napoli, via Vivaldi, Caserta, Italy,}
\email{giuseppina.diblasio@unina2.it}
\author{Bruno Volzone}
\address[BRUNO VOLZONE]{Universit\`a degli Studi di Napoli ``Parthenope'', Facolt\`a di Ingegneria, Dipartimento per le Tecnologie, Centro Direzionale Isola C/4 80143 Napoli, Italy.}
\email{bruno.volzone@uniparthenope.it}
\date{}
\begin{abstract}
In this paper we establish a comparison result through symmetrization for solutions to some boundary value problems involving the fractional Laplacian. This allows to get  sharp estimates for the solutions, obtained by comparing them with solutions of suitable radial problems. Furthermore, we use such result to prove \emph{a priori} estimates for solutions in terms of the data, providing several regularity results which extend the well known ones for the classical Laplacian.
\end{abstract}

\maketitle
\

\section{Introduction and main results}

The final goal of this paper is to obtain a comparison principle using symmetrization techniques
in order to get sharp estimates for solutions to some elliptic boundary value
problems involving the fractional Laplacian operator. Let $u:{\mathbb{R}}^{N}\rightarrow{\mathbb{R}}$ be in the Schwartz space of rapidly decaying functions, here the
fractional Laplacian $(-\Delta)^{\alpha/2}$ of $u$  with $\alpha\in(0,2)$, is
defined in a standard sense, that is either by the Riesz potential
\[
(-\Delta)^{\alpha/2}u(x):=C_{N,\alpha}\text{ P. V.}\int_{{\mathbb{R}}^{N}%
}\frac{u(x)-u(\xi)}{|x-\xi|^{N+\alpha}}d\xi,
\]
where P.V. denotes the principal value and $C_{N,\alpha}$ is a suitable
normalization constant, or as a pseudo differential operator through the
Fourier transform on the Schwartz class
\[
\mathcal{F}[(-\Delta)^{\alpha/2}u](\xi):=|\xi|^{\alpha}\mathcal{F}[u](\xi)
\]
where $\mathcal{F}[g]$ denotes the Fourier transform of a function $g$. As for
the equivalence between these two notions, as well as a detailed description
and properties concerning more general integro-differential operators, we
refer to the book of Landkof \cite{LANDKOF} and the paper \cite{Silv2}%
.\newline It is well known (see for instance \cite{Caffext}) that to any
function $u$ smooth enough on ${\mathbb{R}}^{N}$ we can associate its
$\alpha$-harmonic extension, namely a function $w$ defined on the upper
half-space ${\mathbb{R}}_{+}^{N+1}:={\mathbb{R}}^{N}\times(0,+\infty)$ which
is a solution to the local (degenerate or singular) elliptic problem
\begin{equation}
\left\{
\begin{array}
[c]{lll}%
-\operatorname{div}\left(  y^{1-\alpha}\,\nabla w\right)  =0 &  & in\text{
}{\mathbb{R}}_{+}^{N+1}\\
&  & \\
w(x,0)=u(x) &  & in\text{ }{\mathbb{R}}^{N}.
\end{array}
\right.  \label{ext}%
\end{equation}
Moreover, Caffarelli and Silvetre give in\ \cite{Caffext} an
interpretation of the fractional Laplacian $(-\Delta)^{\alpha/2}$ as a
Dirichlet to Neumann map:
\begin{equation}
(-\Delta)^{\alpha/2}u(x)=-\frac{1}{k_{\alpha}}\lim_{y\rightarrow0^{+}%
}y^{1-\alpha}\,\frac{\partial w}{\partial y}(x,y),\label{DirNeum}
\end{equation}
where $k_{\alpha}$ is a suitable constant, whose exact value is
\begin{equation}
\kappa_{\alpha}=\frac{2^{1-\alpha}\Gamma(1-\frac{\alpha}{2})}{\Gamma
(\frac{\alpha}{2})}.\label{kappa}%
\end{equation}
In order to define the fractional Laplacian in
bounded domains $\Omega$ the above characterization has to be suitably
adapted. This has been done in the papers \cite{Colorado} and \cite{Cab2}, where formula (\ref{DirNeum}) allows to \emph{define} the fractional Laplacian $(-\Delta)^{\alpha/2}$ over a proper function space on $\Omega$, as we shall see in Section 2. \newline The
fractional Laplacian appears in several contexts. For instance, it arises in
the study of various physical phenomena, where long-range or anomalous
diffusions occur. Just to give few examples, this kind of operator can be found in combustion theory (see \cite{Caffree2}), in
dislocations processes of mechanical systems (see \cite{Imb2}) or in
crystals (see \cite{Garroni}). Moreover, as it is well known in the theory of probability,
the fractional Laplacian is the infinitesimal generator of a Lévy process (see
for instance \cite{Sato}). Due to all of that, lots of authors devoted their
interest to the subject. We just mention \cite{Silv2}, \cite{Sil3},
\cite{Caffree}, \cite{Caffree2} dedicated to the obstacle problem and the free
boundaries for the fractional Laplacian, the papers \cite{Cab2}, \cite{Cab3} regarding
some aspects of nonlinear equations involving fractional powers of the
Laplacian, the convex-concave problem for the fractional Laplacian described
in \cite{Colorado}, the work \cite{Capella} in which a critical exponent
problem for the half-Laplacian in an annulus is investigated, the study \cite{Vald} of a nonlocal energy variational problem,
and the papers \cite{Bog2}, \cite{Bog4}, \cite{Bog5}, \cite{Bog6}. Obviously
this list is very far from being exhaustive. \smallskip\\In order to describe our main result let us consider the nonlocal
Dirichlet problem with homogeneous boundary condition
\begin{equation}
\left\{
\begin{array}
[c]{lll}%
\left(  -\Delta\right)  ^{\alpha/2}u=f\left(  x\right)   &  & in\text{ }%
\Omega\\
&  & \\
u=0 &  & on\text{ }\partial\Omega,
\end{array}
\right.  \label{eq.1}%
\end{equation}
where $\Omega$ is
an open bounded set of ${\mathbb{R}}^{N}$ and $f$ is a smooth function on $\Omega$. To this problem it is possible to associate
the local one
\begin{equation}
\left\{
\begin{array}
[c]{lll}%
-\operatorname{div}\left(  y^{1-\alpha}\nabla w\right)  =0 &  & in\text{
}\mathcal{C}_{\Omega}\\
&  & \\
w=0 &  & on\text{ }\partial_{L}\mathcal{C}_{\Omega}\\
&  & \\
-\dfrac{1}{\kappa_{\alpha}}\lim\limits_{y\rightarrow0^{+}}y^{1-\alpha}%
\dfrac{\partial w}{\partial y}=f\left(  x\right)   &  & in\text{ }\Omega,
\end{array}
\right.  \label{eq.2}%
\end{equation}
where $\mathcal{C}_{\Omega}:=\Omega\times\left(  0,+\infty\right)  $ is the
cylinder of basis $\Omega$ and $\partial_{L}\mathcal{C}_{\Omega}%
:=\partial\Omega\times\lbrack0,+\infty)$ is its lateral boundary. We recall that (see for instance \cite{Colorado}) given a solution $w$ in the weak sense to problem (\ref{eq.2}), then its trace on
$\Omega$, $tr_{\Omega}(w)=w(\cdot,0)=:u$ is a solution to
problem (\ref{eq.1}) (see also Section 3 for precise definitions).\\
Following \cite{Talenti}, the idea is to get sharp estimates for the solution
$u$ to (\ref{eq.1}) by comparing it with a solution $\phi$ to the
radial problem
\begin{equation}
\left\{
\begin{array}
[c]{lll}%
\left(  -\Delta\right)  ^{\alpha/2}\phi=f^{\#}\left(  x\right)  &  & in\text{
}\Omega^{\#}\\
&  & \\
\phi=0 &  & on\text{ }\partial\Omega^{\#},
\end{array}
\right.  \label{sym}%
\end{equation}
where $\Omega^{\#}$ is the ball centered at 0, having the same measure as
$\Omega$ and $f^{\#}$ is the Schwarz rearrangement of $f$.
Since $\phi$ is the trace on $\Omega^{\#}$ of a solution $v$ to the
problem
\begin{equation}
\left\{
\begin{array}
[c]{lll}%
-\operatorname{div}\left(  y^{1-\alpha}\nabla v\right)  =0 &  & in\text{
}\mathcal{C}_{\Omega}^{\#}\\
&  & \\
v=0 &  & on\text{ }\partial_{L}\mathcal{C}_{\Omega}^{\#}\\
&  & \\
-\dfrac{1}{\kappa_{\alpha}}\lim\limits_{y\rightarrow0^{+}}y^{1-\alpha}%
\dfrac{\partial v}{\partial y}=f^{\#}\left(  x\right)  &  & in\text{ }%
\Omega^{\#},
\end{array}
\right.  \label{eq.4}%
\end{equation}
where $\mathcal{C}_{\Omega}^{\#}:=\Omega^{\#}\times\left(  0,+\infty\right)
$, $\partial_{L}\mathcal{C}_{\Omega^{\#}}:=\partial\Omega^{\#}\times
\lbrack0,+\infty)$, it makes sense to look for a
comparison between concentrations of the functions $w$ and $v$ through their Schwarz rearrangements (see
Section 2 for definitions). More precisely, we prove that
\begin{equation}\label{compar}
\int_{0}^{s}w^{\ast}\left(  \sigma,y\right)  d\sigma\leq\int_{0}^{s}v^{\ast
}\left(  \sigma,y\right)  d\sigma\quad\forall s\in\left[  0,\left\vert
\Omega\right\vert \right]
\end{equation}
where $w^{\ast}(\cdot,y)$, $v^{\ast}(\cdot,y)$ are the one dimensional rearrangements of $w\,,v$ respectively, for any fixed $y\in\left[  0,+\infty\right)  .$
The achievement of such result looks reasonable because of the nature of problem (\ref{eq.2}), for which a symmetrization with
respect to $x$ keeping the $y$ variable fixed (\emph{i.e.} Steiner
symmetrization with respect to the line $x=0$) is available. The key role in
this framework is played by a second order derivation formula for
functions defined by integrals, obtained in \cite{MR1374171} for the smooth
case and in \cite{MR1649548} for less regular functions.\\ We point out that through inequality (\ref{compar}) we easily get a comparison result between the traces on $\Omega\times\{0\}$ of $w$ and $v$,
namely an integral comparison between $u$ and $\phi$:
\begin{theorem}
Let $u$ and $\phi$ be the weak solutions to problems
\emph{(\ref{eq.1})} and \emph{(\ref{sym})}, respectively, and $f\in
L^{\frac{2N}{N+\alpha}}(\Omega)$, with $\alpha\in(0,2).$ Then we have:%
\[
\int_{0}^{s}u^{\ast}\left(  \sigma\right)  d\sigma\leq\int_{0}^{s}\phi^{\ast
}\left(  \sigma\right)  d\sigma\quad\forall s\in\left[  0,\left\vert
\Omega\right\vert \right]  .
\]
\end{theorem}

We emphasize that since for $\alpha=2$ the fractional Laplacian coincides with the classical Laplacian,
for which comparison and regularity results via symmetrization methods are well known (see \emph{e.g.} \cite{Talenti}, \cite{Talenti2}, \cite{Al2}), in the following we consider only the  case $0<\alpha<2$.\smallskip\\
In this setting the comparison result proved in Theorem 1.1 allows us to prove \emph{a priori} estimates for solutions of problem (\ref{eq.1}) in terms of the data $f$, providing several regularity results which extend the well known ones for the classical Laplacian (see Section 5).
\newline We also want to remark that the application of symmetrization techniques to general
Lévy processes is not new, as it is shown for example in \cite{Betsakos} and
\cite{SymmLevy} where several isoperimetric-type issues are investigated.
Moreover, our approach
is completely \textquotedblleft PDE oriented\textquotedblright\ and it is not
based on a probabilistic setting.

This paper is organized as follows. In Section 2 we give some basic
definitions and properties concerning the functional setting we are going to
work with. In particular, in Subsection 2.1 we recall some fundamental definitions concerning the fractional Laplacian in bounded
domains and properties of weak solutions to the related Dirichlet problem.
Subsections 2.2 and 2.3 introduce the notion of Schwarz rearrangement and Lorentz space, together with some related properties. In Subsection 2.4 we define the Green function for the fractional Laplacian, an essential tool to write down an explicit integral
representation of the solution $\phi$ to the radial problem \eqref{sym}.
In Section 3 we prove the comparison results stated above. Furthermore in Section
4, we exhibit some regularity results of the solution $u$ in terms of the source
data $f$.
Finally in Section 5 we give some comments about the best constant in the
$L^{\infty}$ estimate. In particular we explicitly compute it for $u$ on a
ball, considering only the case $\alpha=1$ and $N=3$.\newline

\section{Preliminaries}

\subsection{Function Spaces and Definitions}

As we pointed out in the introduction, formula \eqref{DirNeum} given in \cite{Caffext} connects the nonlocal character of
$\left(  -\Delta\right)  ^{\alpha/2}$  to local problems of the form \eqref{ext}. This interpretation can be extended to the case of bounded domains. To this aim, it is convenient to introduce here a suitable functional setting and
basic definitions. For all the details and proofs of the following definitions and properties, we refer to the
the papers \cite{Cab2}, \cite{Colorado}, \cite{Cab3} and
\cite{Capella2}.\smallskip\\
If $\Omega$ is a bounded domain in $\mathbb{R}^{N}$,
the half-cylinder with base $\Omega$ and its lateral boundary will be respectively denoted by
\[
\mathcal{C}_{\Omega}:=\Omega\times\left(  0,+\infty\right)  \qquad\text{and}\qquad\partial
_{L}\mathcal{C}_{\Omega}:=\partial\Omega\times\lbrack0,+\infty).
\]
We introduce then the weighted energy space
\[
X_{0}^{\alpha}(\mathcal{C}_{\Omega}):=\left\{  w\in H^{1}(\mathcal{C}_{\Omega
}),\,w=0\,\text{ on }\partial_{L}\mathcal{C}_{\Omega}\,\,:\int_{\mathcal{C}%
_{\Omega}}y^{1-\alpha}|\nabla w(x,y)|^{2}\,dxdy<\infty\right\}
\]
equipped with the norm
\[
\Vert w\Vert_{X_{0}^{\alpha}}:=\left(  \int_{\mathcal{C}_{\Omega}}y^{1-\alpha
}|\nabla w(x,y)|^{2}\,dxdy\right)  ^{1/2}.
\]
Thus we define the \emph{trace space} by
\begin{equation}
\mathcal{V_{\alpha}}(\Omega)=\left\{  u=tr_{\Omega}w:=w(\cdot,0)\,:\,w\in
X_{0}^{\alpha}(\mathcal{C}_{\Omega})\right\}  , \label{trace space}%
\end{equation}
where $tr_{\Omega}$ is the \emph{trace} operator on the space $w\in X_{0}^{\alpha}(\mathcal{C}_{\Omega})$.\\
The fractional Laplacian in a bounded domain $\Omega$ is well defined for
function in $\mathcal{V_{\alpha}}(\Omega).$ Indeed it is well known (see e.g.
\cite{Cab2}, \cite{Colorado}) that for any function $u\in\mathcal{V_{\alpha}%
}(\Omega)$ there exists a unique minimizer $w$ to the problem
\[
\inf\left\{  \int_{\mathcal{C}_{\Omega}}y^{1-\alpha}\,|\nabla w(x,y)|^{2}%
\,dx\,dy\,:\,w\in X_{0}^{\alpha}(\mathcal{C}_{\Omega}),\,w(\cdot
,0)=u\,\,\text{in }\Omega\right\}.
\]
By standard elliptic theory such minimizer $w$ is smooth for $y>0$ and
satisfies
\begin{equation}
\left\{
\begin{array}
[c]{lll}%
-\operatorname{div}\left(  y^{1-\alpha}\nabla w\right)  =0 &  & in\text{
}\mathcal{C}_{\Omega}\\
&  & \\
w=0 &  & on\text{ }\partial_{L}\mathcal{C}_{\Omega}\\
&  & \\
w(\cdot,0)=u &  & in\text{ }\Omega.
\end{array}
\right.  \label{Prob_Ext(u)}%
\end{equation}
This yields to consider an \emph{extension operator} in the following sense:

\begin{definition}
Given a function $u\in\mathcal{V_{\alpha}}(\Omega)$, the solution $w$ to problem (\ref{Prob_Ext(u)}) will be said the $\alpha$-harmonic extension of $u$ on the cylinder $\mathcal{C}_{\Omega}$ and will be denoted by $\emph{Ext}_{\alpha}u$.
\end{definition}
Then the fractional Laplacian operator can be defined through the \emph{Dirichlet to
Neumann map} as follows (see e.g. \cite{Cab2}, \cite{Colorado}):

\begin{definition}
\label{fracLap}For any $u\in\mathcal{V_{\alpha}}(\Omega)$ we define the
fractional Laplacian $(-\Delta)^{\alpha/2}$ acting on $u$ as the following
limit (in the distributional sense)
\[
(-\Delta)^{\alpha/2}u(x):=-\frac{1}{\kappa_{\alpha}}\lim_{y\rightarrow
0}y^{1-\alpha}\frac{\partial w}{\partial y}(x,y),
\]
where $w=$Ext$_{\alpha}(u)$ and
$\kappa_{\alpha}$ is given by \emph{(\ref{kappa})}.
\end{definition}

Let $\left\{  \varphi_{k}\right\}  $ be an orthonormal basis of $L^{2}\left(
\Omega\right)$ made by eigenfunctions of $-\Delta$ in $\Omega$
with zero Dirichlet boundary conditions and $\left\{  \lambda_{k}\right\}  $
the corresponding Dirichlet eigenvalues. It is classical that the powers of a
positive operator in a bounded domain, evaluated on certain function $u$,
are defined through the spectral decomposition of $u$ using
the powers of the eigenvalues of the original operator. So in the case of the fractional Laplacian $(-\Delta)^{\alpha/2}$, if $$u=\underset{k=1}{\overset{\infty}{\sum}}a_{k}\varphi_{k}$$ we must have
\begin{equation}
(-\Delta)^{\alpha/2} u  =\underset{k=1}{\overset{\infty}{\sum}%
}a_{k}\lambda_{k}^{\alpha/2}\varphi_{k}.\label{DirNeuboun}
\end{equation}
This definition is coherent with Definition \ref{fracLap}, since it is possible to
give the following characterization of the trace space $\mathcal{V_{\alpha}}(\Omega)$:

\begin{proposition}
The space $\mathcal{V_{\alpha}}(\Omega)$ defined in (\ref{trace space})
coincides with the space
\begin{equation}
H:=\left\{  u\in L^{2}\left(  \Omega\right)  \mid u=\underset{k=1}%
{\overset{\infty}{\sum}}a_{k}\varphi_{k}\text{ satisfying }\underset
{k=1}{\overset{\infty}{\sum}}a_{k}^{2}\lambda_{k}^{\alpha/2}<\infty\right\}  .
\label{def H}%
\end{equation}
Moreover if $u\in\mathcal{V_{\alpha}}(\Omega)$ admits the decomposition
$u=\sum_{k=1}^{\infty}a_{k}\,\varphi_{k}$, then its $\alpha$- harmonic
has the following explicit representation
\begin{equation}
\emph{Ext}_{\alpha}u\left(  x,y\right)  =\underset{k=1}{\overset{\infty
}{\sum}}a_{k}\varphi_{k}\left(  x\right)  \rho(\lambda_{k}^{1/2}y)\label{Extens}
\end{equation}
where $\rho$ solves the problem
\[
\left\{
\begin{array}
[c]{lll}%
\rho^{\prime\prime}(s)+\dfrac{1-\alpha}{s}\rho^{\prime}\left(  s\right)
=\rho\left(  s\right)  &  & s>0\\
&  & \\
\lim\limits_{y\rightarrow0^{+}}y^{1-\alpha}\rho^{\prime}\left(  s\right)
=-\kappa_{\alpha} &  & \\
&  & \\
\rho\left(  0\right)  =1. &  &
\end{array}
\right.
\]
\end{proposition}
Therefore, using \eqref{Extens} and Definition 2.2, equality \eqref{DirNeuboun} easily follows.\smallskip\\
According to \cite{Colorado}, we have the following definition of weak solution to problems of the type \eqref{eq.2}:

\begin{definition}
Let $f\in L^{\frac{2N}{N+\alpha}}(\Omega)$, where $\alpha\in(0,2)$. We say
that $w\in X_{0}^{\alpha}(\mathcal{C}_{\Omega})$ is the weak solution to problem \emph{(\ref{eq.2})} if for any test function
$\varphi\in X_{0}^{\alpha}(\mathcal{C}_{\Omega})$ the following identity
holds:
\begin{equation}
\int_{\mathcal{C}_{\Omega}}y^{1-\alpha}\,\nabla w(x,y)\cdot\nabla
\varphi(x,y)\,dxdy=\kappa_{\alpha}\int_{\Omega}f(x)\,\varphi(x,0)dx.
\label{weakform}%
\end{equation}

\end{definition}

We note that for any test function $\varphi\in X_{0}^{\alpha}(\mathcal{C}_{\Omega})$, by the
Sobolev trace inequality (see
\cite{Colorado}) it follows that the trace
$\varphi(\cdot,0)$ on $\Omega\times\left\{  0\right\}  $ belongs to
$L^{\frac{2N}{N-\alpha}}(\Omega)$, hence the integral at the right-hand side
of \eqref{weakform} makes sense. Besides, the classical Lax Milgram theorem
ensures that a unique weak solution $w\in X_{0}^{\alpha}(\mathcal{C}_{\Omega
})$ to problem (\ref{eq.2}) exists.\smallskip\\Then the definition of weak solution to problem \eqref{eq.1} is strictly related
to the solution of \eqref{eq.2} in the following sense:

\begin{definition}
Let $f\in L^{\frac{2N}{N+\alpha}}(\Omega)$, where $\alpha\in(0,2)$. We say
that $u\in H$ is the weak solution to \eqref{eq.1} if $u=tr_{\Omega}w$, and
$w$ is the weak solution to problem \eqref{eq.2}.
\end{definition}

We observe that if $u$ is the weak solution to \eqref{eq.1}, than its $\alpha$
harmonic extension $\text{Ext}_{\alpha}u$ is smooth for $y>0$ and
decays to zero as $y\rightarrow\infty$ (see \cite{Colorado}).\smallskip\\
Finally we point out that the space $H$ defined in (\ref{def H}) is an
interpolation space and it is possible to prove that (see \cite{Cab2} for the case $\alpha=1$ and \cite{MR0247243},
\cite{Capella2} for the general case)
\[
H=\left\{
\begin{array}
[c]{lll}%
H^{\alpha/2}(\Omega) &  & if\text{ }\alpha\in\left(  0,1\right) \\
&  & \\
H_{00}^{1/2}(\Omega) &  & if\text{ }\alpha=1\\
&  & \\
H_{0}^{\alpha/2}(\Omega) &  & if\text{ }\alpha\in\left(  1,2\right)
\end{array}
\right.
\]
where $H^{\alpha/2}(\Omega)$ is the usual fractional Sobolev space, $H_{0}^{\alpha/2}(\Omega)$ is the closure of $C_{0}^{\infty}\left(
\Omega\right)  $ with respect to the norm $\left\Vert \cdot\right\Vert
_{H^{\alpha/2}(\Omega)}$ and%
\[
H_{00}^{1/2}(\Omega):=\left\{  u\in H^{1/2}(\Omega):\int_{\Omega}%
\frac{u(x)^{2}}{d\left(  x\right)  }dx<\infty\right\}  ,
\]
with $d(x):=\text{dist}(x,\partial\Omega)$. \newline

\subsection{Basic facts about rearrangements}

Let $\Omega$ be a bounded open subset of $%
\mathbb{R}
^{N}$ and $u$ be a real measurable function on $\Omega$. We will denote by $\left\vert
\cdot\right\vert $ the $N$-dimensional Lebesgue measure. We define the
\emph{distribution function} $\mu_{u}$ of $u$ as%
\[
\mu_{u}\left(  t\right)  =\left\vert \left\{  x\in\Omega:\left\vert u\left(
x\right)  \right\vert >t\right\}  \right\vert \text{ ,
}t\geq0,
\]
and the \emph{decreasing rearrangement} of $u$ as%
\[
u^{\ast}\left(  s\right)  =\sup\left\{  t\geq0:\mu_{u}\left(  t\right)
>s\right\}  \text{ , }s\in\left(  0,\left\vert \Omega\right\vert \right)
\]
Furthermore, if $\omega_{N\text{ }}$ is the measure of the unit ball in $%
\mathbb{R}
^{N}$ and $\Omega^{\#}$ is the ball of $%
\mathbb{R}
^{N}$ centered at the origin having the same Lebesgue measure as $\Omega,$ the
function%
\[
u^{\#}\left(  x\right)  =u^{\ast}(\omega_{N}\left\vert x\right\vert
^{N})\text{ \ , }x\in\Omega^{\#}%
\]
is called \emph{decreasing spherical rearrangement} of $u$. For an exhaustive
treatment of rearrangements we refer to \cite{Band}, \cite{Kawohl} and to the
appendix of \cite{Talenti2}. Here we just recall the well known
Hardy-Littlewood inequality (see \cite{MR0046395})%
\begin{equation}
\int_{\Omega}\left\vert u\left(  x\right)  v\left(
x\right)  \right\vert dx\leq\int_{0}^{\left\vert \Omega\right\vert }u^{\ast
}\left(  s\right)  v^{\ast}\left(  s\right)  ds=\int_{\Omega^{\#}}u^{\#}(x)\,v^{\#}(x)\,dx
\label{HardyLit}%
\end{equation}
where $u,v$ are measurable functions on $\Omega$.\smallskip\newline We point
out that as we will deal with two variable functions of the type
\begin{equation}
u:\left(  x,y\right)  \in\mathcal{C}_{\Omega}\rightarrow u\left(  x,y\right)
\in{\mathbb{R}} \label{u}%
\end{equation}
defined on the cylinder  $\mathcal{C}%
_{\Omega}:=\Omega\times\left(  0,+\infty\right)  $, measurable with respect to
$x,$ we can define the Steiner symmetrization of $\mathcal{C}_{\Omega}$ with
respect to the variable $x$, namely the set
\hbox{$\mathcal{C}_{\Omega}^{\#}:=\Omega
^{\#}\times\left(  0,+\infty\right).$} In addition, we will denote by $\mu
_{u}\left(  t,y\right)  $ and $u^{\ast}\left(  s,y\right)  $ the distribution
function and the decreasing rearrangements of (\ref{u}), with respect to $x$
for $y$ fixed, and we define the function%
\[
u^{\#}\left(  x,y\right)  =u^{\ast}(\omega_{N}|x|^{N},y)
\]
which is the \emph{Steiner symmetrization} of $u$, with respect to the line
$x=0.$ Obviously $u^{\#}$ is a spherically symmetric and decreasing function
with respect to $x$ for any fixed $y$.\smallskip\\
Now we recall two derivations formulas, that will turn out very useful in the proof of the main result. The following proposition can be found in \cite{Mossino}, and it is a generalization of a well-known result by Bandle (see \cite{Band}).
\begin{proposition}
Suppose that $w\in H^{1}(0,T;L^{2}(\Omega))$ for some $T>0$. Then $$w^{*}\in H^{1}(0,T;L^{2}(0,|\Omega|))$$ and if $|\left\{w(x,t)=w^{*}(s,t)\right\}|=0$ for a.e. $(s,t)\in(0,|\Omega|)\times(0,T)$, the following derivation formula occurs
\begin{equation}
\int_{w(x,y)>w^{*}(s,y)}\frac{\partial w}{\partial y}(x,y)\,dx=\int_{0}^{s}\frac{\partial w^{*}}{\partial y}(s,y)\,ds.\label{Rakotoson}
\end{equation}
\end{proposition}

Moreover, what follows is a second order derivation formula due to Mercaldo and Ferone
(see \cite{MR1649548}), which is a suitable generalization of that contained in
\cite{MR1374171}, where only analytic functions are considered.

\begin{proposition}
\label{Ferone-Mercaldo} Let $w\in W^{2,\infty}\left(  \mathcal{C}_{\Omega}\right)  $. Then for almost every $y\in(0,+\infty)$ the following
derivation formula holds:
\begin{align*}
\int_{w\left(  x,y\right)  >w^{\ast}\left(  s,y\right)  }\frac{\partial^{2}%
w}{\partial y^{2}}\left(  x,y\right)  dx  &  =\frac{\partial^{2}}{\partial
y^{2}}\int_{0}^{s}w^{\ast}\left(  \sigma,y\right)  d\sigma-\int_{w\left(
x,y\right)  =w^{\ast}\left(  s,y\right)  }\frac{\left(  \frac{\partial
w}{\partial y}\left(  x,y\right)  \right)  ^{2}}{\left\vert \nabla
_{x}w\right\vert }d\mathcal{H}^{N-1}\left(  x\right) \\
&  \!\!\!+\left(  \int_{w\left(  x,y\right)  =w^{\ast}\left(  s,y\right)
}\!\frac{\frac{\partial w}{\partial y}\left(  x,y\right)  }{\left\vert
\nabla_{x}w\right\vert }\!d\mathcal{H}^{N-1}\left(  x\right)  \!\right)
^{2}\!\left(  \!\int_{w\left(  x,y\right)  =w^{\ast}\left(  s,y\right)
}\!\frac{1}{\left\vert \nabla_{x}w\right\vert }\!d\mathcal{H}^{N-1}\left(
x\right)  \!\right)  ^{-1}\!.
\end{align*}

\end{proposition}
\smallskip
\subsection{Lorentz spaces}

As we will deal with some sharp regularity results of the solution $u$ to
\eqref{eq.1} in terms of the data $f$, we introduce here basic notions
regarding the functional spaces where $f$ will be supposed to belong
to.\newline Let $\Omega$ be a bounded open set of $\mathbb{R}^{N}$. We say
that a measurable function $u:\Omega\rightarrow\mathbb{R}$ belongs to the
Lorentz\ $L^{p,q}\left(  \Omega\right)  $\ for $0<p,q\leq+\infty$ if the
quantity
\begin{equation}
||u||_{L^{p,q}(\Omega)}=\left\{
\begin{array}
[c]{ll}%
\left(
{\displaystyle\int_{0}^{+\infty}}
\left[  t^{\frac{1}{p}}u^{\ast}(t)\right]  ^{q}\frac{dt}{t}\right)  ^{\frac
{1}{q}} & 0<q<\infty\\
\underset{0<t<+\infty}{\sup}t^{\frac{1}{p}}u^{\ast}(t) & q=\infty
\end{array}
\right.  \label{def lor-zig}%
\end{equation}
is finite. We remark that for $p>1$, and $q\geq1$, the quantity in
(\ref{def lor-zig}) can be equivalently defined replacing $u^{\ast}\left(
t\right)  $ with $$u^{\ast\ast}\left(  t\right)  =\frac{1}{t}\int_{0}%
^{t}u^{\ast}(s)\, ds.$$
We stress that the $L^{p,q}-$norm, for every $1<p,q\leq+\infty,$ is
rearrangement invariant, that is%
\[
\left\Vert u\right\Vert _{L^{p,q}\left(  \Omega\right)  }=\Vert u^{\#}%
\Vert_{L^{p,q}\left(  \Omega^{\#}\right)  }.
\]
Besides, we emphasize that $L^{p,q}\left(  \Omega\right)  =L^{p}(\Omega)$,
$L^{p,\infty}\left(  \Omega\right)  =\mathcal{M}_{p}$ (the Marcinkiewicz
space) for any $1\leq p\leq\infty$ and, for $1<q<p<r<\infty$ the following
inclusion occurs:
\[
L^{\infty}(\Omega)\subset L^{r}(\Omega)\subset L^{p,1}(\Omega)\subset
L^{p,q}(\Omega)\subset L^{p,p}(\Omega)=L^{p}(\Omega)\subset L^{p,r}%
(\Omega)\subset L^{p,\infty}(\Omega) \subset L^{q}(\Omega).
\]
Now we recall a convolution inequality in Lorentz spaces due to O'Neil (\cite{O'Neil}), which will be an essential tool to obtain some  \emph{a priori} estimates for solutions to problems of the type \eqref{eq.1} in terms of data belonging to Lorentz spaces (see Theorem 4.3):
\begin{theorem}
Suppose that $f\in L^{p_1,q_1}\left(  \R^{N}\right)\,,g\in L^{p_2,q_2}\left(  \R^{N}\right)$ where
\[
\frac{1}{p_1}+\frac{1}{p_2}>1.
\]
Then $f\ast g\in L^{r,s}\left(  \R^{N}\right)$ where
\[
\frac{1}{p_1}+\frac{1}{p_2}-1=\frac{1}{s},
\]
and $t\geq1$ is any number such that
\[
\frac{1}{q_1}+\frac{1}{q_2}\geq\frac{1}{t}.
\]
Moreover
\[
||f\ast g||_{L^{s,t}(\R^{N})}\leq 3s\, ||f||_{L^{p_1,q_1}(\R^{N})}\,||g||_{L^{p_2,q_2}(\R^{N})}.
\]
\end{theorem}

\subsection{Spectral decomposition of the solution}

In this section we highlight some properties concerning the representation of the solution to the fractional Poisson equation by the Green function and the link with its spectral decomposition.
According to what we have said in Subsection 2.1, it is always possible to get a spectral decomposition of the solution $u$ to \eqref{eq.1} in terms of the Fourier coefficients of the source term $f$. Indeed, suppose that $\left\{  \varphi_{k}\right\}  $ is an orthonormal basis of $L^{2}\left(
\Omega\right)$ made by eigenfunctions of $-\Delta$ in $\Omega$
with zero Dirichlet boundary conditions and $\left\{  \lambda_{k}\right\}  $
the corresponding Dirichlet eigenvalues.
Therefore, if $u\in H$ is the weak solution to problem
\eqref{eq.1}, having the decomposition

\begin{equation}u=\sum_{k=1}^{\infty}a_{k}%
\,\varphi_{k},\label{decomposit}
\end{equation}
then the fractional Laplacian of $u$ has the
spectral decomposition \eqref{DirNeuboun}.
Thus if $$f=\sum_{k=1}^{\infty}c_{k}\varphi_{k},$$ where $c_{k}%
=(f,\varphi_{k})_{L^{2}(\Omega)}$ are the Fourier coefficient of $f$, the
Fourier coefficients of $u$ are
\begin{equation}
a_{k}=\frac{c_{k}}{\lambda_{k}^{\alpha/2}}. \label{ak}%
\end{equation}
Now, let us denote by $\mathcal{G}_{D}(x,y)$\ the \emph{Green function of a bounded domain}
$D\subseteq{\mathbb{R}}^{N}$ for the fractional Laplacian $(-\Delta)^{\alpha/2}$. Then we have (see \cite{LANDKOF}, \cite{Bogdan1}, \cite{Blument})
\begin{equation}
-(-\Delta)_{x}^{\alpha/2}\mathcal{G}_{D}(x,y)=\delta(x-y)\quad\text{in }\mathcal{D}^{\prime}(D).\label{deltagre}
\end{equation}
Next, suppose that the function $\mathcal{G}_{D}$ has
the following expansion, for any fixed $y\in D$:
\[
\mathcal{G}_{D}(x,y)=\sum_{k=1}^{\infty}c_{k}(y)X_{k}(x).
\]
Then equality \eqref{DirNeuboun} provides the following spectral decomposition for the fractional Laplacian
of $\mathcal{G}_{D}$:
\begin{equation}
(-\Delta)_{x}^{\alpha/2}\mathcal{G}_{D}(x,y)=\sum_{k=1}^{\infty
}\lambda_{k}^{\alpha/2}c_{k}(y)\varphi_{k}(x). \label{Green2}%
\end{equation}
If we multiply both sides of equation \eqref{deltagre} by $\varphi_{m}$ and integrate
over $D$ with respect to $x$, equation (\ref{Green2}) links to
\[
\sum_{k=1}^{\infty}\lambda_{k}^{\alpha/2}c_{k}(y)\int_{D}%
\varphi_{k}(x)\varphi_{m}(x)dx=-\varphi_{m}(y)
\]
i.e.
\[
c_{m}(y)=-\frac{\varphi_{m}(y)}{\lambda_{m}^{\alpha/2}}%
\]
that is
\begin{equation}
\mathcal{G}_{D}(x,y)=-\sum_{k=1}^{\infty}\frac{\varphi_{k}(x)\,\varphi_{k}%
(y)}{\lambda_{k}^{\alpha/2}}. \label{Green3}%
\end{equation}
Hence from \eqref{decomposit}, \eqref{ak} and \eqref{Green3}  we easily infer that
\begin{equation}
u=\sum_{k=1}^{\infty}\frac{\varphi_{k}(x)}{\lambda_{k}^{\alpha/2}}\int_{D}f(y)\,\varphi_{k}(y)\,dy
=-\int_{D}\mathcal{G}_{D}(x,y)\,f(y)\,dy.\label{intrep}
\end{equation}
When
$D$ is a ball $B(0,R)$, we shall frequently use the following explicit expression of the Green function (see \cite{Blument},
\cite{LANDKOF}, \cite{Bogdan1})
\begin{equation}
\mathcal{G}_{B(0,R)}(x,y)=-2^{-\alpha}\frac{\Gamma\left(  \frac{N}{2}\right)
}{R^{2}\pi^{N/2}}\Gamma\left(  \frac{\alpha}{2}\right)  ^{-2}\,|x-y|^{\alpha
-N}\,%
{\displaystyle\int_{0}^{z}}
\frac{s^{\frac{\alpha}{2}-1}}{\left(  \frac{s}{R^{2}}+1\right)  ^{N/2}}\,ds
\label{GreenB(0,R)}%
\end{equation}
where $x,\,y\in B(0,R)$ and
\[
z=\frac{(R^{2}-|x|^{2})(R^{2}-|y|^{2})}{|x-y|^{2}}.
\]
We stress that (\ref{GreenB(0,R)}) coincides with the Green
function of classical Laplacian for $\alpha=2$. Clearly we have
\begin{equation}
|\mathcal{G}_{B(0,R)}(x,y)|\leq\frac{\mathsf{a}\,\mathsf{b}}{|x-y|^{N-\alpha}} \label{Est_Green}%
\end{equation}
for $x,y\in B(0,R)$ s. t. $x\neq y$, where
\begin{equation}
\mathsf{a}:=2^{-\alpha}\dfrac{\Gamma\left(  \frac{N}{2}\right)  \,\Gamma
\left(  \frac{\alpha}{2}\right)  ^{-2}}{\pi^{N/2}}R^{N-2},\quad
\mathsf{b}:={\displaystyle\int_{0}^{\infty}\dfrac{s^{\frac{\alpha}{2}-1}%
}{\left(  s+R^{2}\right)  ^{N/2}}\,ds}.\label{eq.35}%
\end{equation}

\section{\bigskip Comparison result}

The aim of this section is to obtain a comparison result between the solutions
of problems (\ref{eq.2}) and (\ref{eq.4}). The symmetrization method allows to
obtain a priori estimates which are the main tools to obtain
regularity results.

\begin{theorem}
\label{Th_comparison_w_v}Let $w$ and $v$ be the the weak solutions to problems
\emph{(\ref{eq.2})} and \emph{(\ref{eq.4})}, respectively, and $f\in
L^{\frac{2N}{N+\alpha}}(\Omega)$, with $\alpha\in(0,2).$Then we have:%
\begin{equation}
\int_{0}^{s}w^{\ast}\left(  \sigma,z\right)  d\sigma\leq\int_{0}^{s}v^{\ast
}\left(  \sigma,z\right)  d\sigma\quad\forall s\in\left[  0,\left\vert
\Omega\right\vert \right]  \label{compres}%
\end{equation}
for any fixed $z\in\left[  0,+\infty\right)  .$
\end{theorem}

\begin{proof}
We first observe that actually there is a clever way to rewrite equation in
problem (\ref{eq.2}), that is
\[
\Delta_{x}w+\frac{1-\alpha}{y}\frac{\partial w}{\partial y}+\frac{\partial
^{2}w}{\partial y^{2}}=0.\label{eq.8}%
\]
As a matter of fact, if we follow \cite{MR2354493} and
make the change of variable%
\[
z=\left(  \frac{y}{\alpha}\right)  ^{\alpha},
\]
we find that problem (\ref{eq.2}) is equivalent to the Cauchy-Dirichlet
problem%
\begin{equation}
\left\{
\begin{array}
[c]{lll}%
z^{\beta}\dfrac{\partial^{2}w}{\partial z^{2}}+\Delta_{x}w=0 &  & in\text{
}\mathcal{C}_{\Omega}\\
&  & \\
w=0 &  & on\text{ }\partial_{L}\mathcal{C}_{\Omega}\\
&  & \\
-\dfrac{\partial w}{\partial z}\left(  x,0\right)  =\kappa_{\alpha}%
\,\alpha^{\alpha-1}f\left(  x\right)  &  & in\text{ }\Omega,
\end{array}
\right.  \label{eq.10}%
\end{equation}
where $\beta:=2\left(  \alpha-1\right)  /\alpha$. The aim is to compare
problem (\ref{eq.10}) with the corresponding symmetrized one:%
\begin{equation}
\left\{
\begin{array}
[c]{lll}%
z^{\beta}\dfrac{\partial^{2}v}{\partial z^{2}}+\Delta_{x}v=0 &  & in\text{
}\mathcal{C}_{\Omega}^{\#}\\
&  & \\
v=0 &  & on\text{ }\partial_{L}\mathcal{C}_{\Omega}^{\#}\\
&  & \\
-\dfrac{\partial v}{\partial z}\left(  x,0\right)  =\kappa_{\alpha}%
\,\alpha^{\alpha-1}f^{\#}\left(  x\right)  &  & in\text{ }\Omega^{\#}.
\end{array}
\right.  \label{prob simm}%
\end{equation}
Now we recall that $w$ is smooth for any $z>0$, so if for a
\emph{fixed} $z>0$ we consider the test function
\[
\varphi_{h}^{z}\left(  x\right)  =\left\{
\begin{array}
[c]{lll}%
\text{sign }(w(x,z)) &  & if\text{ \ }\left\vert w(x,z)\right\vert \geq t+h\\
&  & \\
\dfrac{\left\vert w(x,z)\right\vert -t}{h}\,\text{sign }(w(x,z)) &  & if\text{
\ }t<\left\vert w(x,z)\right\vert <t+h\\
&  & \\
0 &  & if\text{ \ }\left\vert w(x,z)\right\vert \leq t,\text{ }%
\end{array}
\right.
\]
we can multiply the first equation in (\ref{eq.10}) by $\varphi_{h}^{y}\left(
x\right)  $ and integrate over $\Omega$. A simple integration by parts yields
the identity
\[
\frac{1}{h}\int_{t<\left\vert w\right\vert <t+h}\left\vert \nabla
_{x}w\right\vert ^{2}dx-z^{\beta}\frac{1}{h}\int_{\left\vert w\right\vert
>t+h}\dfrac{\partial^{2}w}{\partial z^{2}}dx-z^{\beta}\frac{1}{h}%
\int_{t<\left\vert w\right\vert <t+h}\dfrac{\partial^{2}w}{\partial z^{2}%
}\left(  \dfrac{\left\vert w\right\vert -t}{h}\text{sign }(w)\right)  dx=0
\]
Letting $h\rightarrow0$ and using the isoperimetric inequality, by standard
arguments (see \emph{e.g.} \cite{Talenti})we get%
\[
-z^{\beta}\int_{w\left(  x,z\right)  >t}\dfrac{\partial^{2}w}{\partial z^{2}%
}dx-\left(  \frac{\partial\mu_{w}}{\partial t}\right)  ^{-1}N^{2}\omega
_{N}^{\frac{2}{N}}\left(  \mu_{w}\left(  t\right)  \right)  ^{2-\frac{2}{N}%
}\leq0.
\]
Now if we set
\[
U\left(  s,z\right)  =\int_{0}^{s}w^{\ast}\left(  \sigma,z\right)  d\sigma,
\]
using the second order derivation formula of Proposition \ref{Ferone-Mercaldo}, we find that $U$ verifies the following differential inequality%
\begin{equation}
-z^{\beta}\frac{\partial^{2}U}{\partial z^{2}}-p\left(  s\right)
\frac{\partial^{2}U}{\partial s^{2}}\leq0 \label{U}%
\end{equation}
for a.e. $s\in\left(  0,\left\vert \Omega\right\vert \right)  $ and for any
$z\in\left(  0,+\infty\right)  ,$ where $p\left(  s\right)  =N^{2}$
$\omega_{N}^{\frac{2}{N}}s^{2-\frac{2}{N}}.$ Moreover, the first order
derivation formula \eqref{Rakotoson} implies
\[
\frac{\partial U}{\partial z}=\frac{\partial}{\partial z}\int_{0}^{s}w^{\ast
}\left(  \sigma,z\right)  d\sigma=\frac{\partial}{\partial z}\int_{w\left(
x,z\right)  >w^{\ast}\left(  s,z\right)  }w\left(  x,z\right)  dx=\int
_{w\left(  x,z\right)  >w^{\ast}\left(  s,z\right)  }\frac{\partial
w}{\partial z}\left(  x,z\right)  dx,
\]
\ hence making use of the Hardy-Littlewood inequality \eqref{HardyLit}, we
easily get
\begin{align*}
\dfrac{\partial U}{\partial z}\left(  s,0\right)   &  =%
{\displaystyle\int_{w\left(  x,0\right)  >w^{\ast}\left(  s,0\right)  }}
\dfrac{\partial w}{\partial z}\left(  x,0\right)  dx=-\alpha^{\alpha-1}%
\kappa_{\alpha}%
{\displaystyle\int_{u\left(  x\right)  >u^{\ast}\left(  s\right)  }}
f\left(  x\right)  dx\\
&  \geq-\alpha^{\alpha-1}\kappa_{\alpha}%
{\displaystyle\int_{0}^{s}}
f^{\ast}\left(  \sigma\right)  d\sigma,\quad s\in\left(  0,\left\vert
\Omega\right\vert \right)  .
\end{align*}
So the function $U$ satisfies the following boundary conditions%
\[%
\begin{array}
[c]{l}%
U\left(  0,z\right)  =0\quad\forall z\in\left[  0,+\infty\right) \\
\\
\dfrac{\partial U}{\partial s}\left(  \left\vert \Omega\right\vert ,z\right)
=0\quad\forall z\in\left[  0,+\infty\right) \\
\\
\dfrac{\partial U}{\partial z}\left(  s,0\right)  \geq-\alpha^{\alpha-1}%
\kappa_{\alpha}{\displaystyle\int_{0}^{s}}f^{\ast}\left(  \sigma\right)
d\sigma,\quad s\in\left(  0,\left\vert \Omega\right\vert \right)  .
\end{array}
\]
Now if $v$ is the solution of the symmetrized problem (\ref{prob simm}), being
$v$ radially decreasing with respect to $x$, we obtain
\begin{equation}
-z^{\beta}\frac{\partial^{2}V}{\partial z^{2}}-p\left(  s\right)
\frac{\partial^{2}V}{\partial s^{2}}=0 \label{V}%
\end{equation}
where%
\[
V\left(  s,z\right)  =\int_{0}^{s}v^{\ast}\left(  \sigma,z\right)  d\sigma.
\]
Concerning the boundary conditions, we remark that in this case one has
\begin{align*}
\frac{\partial V}{\partial z}\left(  s,0\right)   &  =-\alpha^{\alpha-1}%
\kappa_{\alpha} \int_{v\left(  \left\vert x\right\vert \right)  >v^{\ast
}\left(  s\right)  } f^{\#}\left(  x\right)  dx\\
&  =-N\omega_{N}\alpha^{\alpha-1}\kappa_{\alpha}\int_{0}^{(s/\omega_{N}%
)^{1/N}} f^{\ast}\left(  \omega_{N}r^{N}\right)  r^{N-1}dr\\
&  =-\alpha^{\alpha-1}\kappa_{\alpha}\int_{0}^{s} f^{\ast}\left(
\sigma\right)  d\sigma\quad s\in\left(  0,\left\vert \Omega\right\vert
\right)
\end{align*}
therefore $V$ satisfies the conditions
\[%
\begin{array}
[c]{l}%
V\left(  0,z\right)  =0\quad\forall z\in\left[  0,+\infty\right) \\
\\
\dfrac{\partial V}{\partial s}\left(  \left\vert \Omega\right\vert ,z\right)
=0\quad\forall z\in\left[  0,+\infty\right) \\
\\
\dfrac{\partial V}{\partial z}\left(  s,0\right)  =-\alpha^{\alpha-1}%
\kappa_{\alpha}\int_{0}^{s} f^{\ast}\left(  \sigma\right)  d\sigma,\quad
s\in\left(  0,\left\vert \Omega\right\vert \right)  .
\end{array}
\]
If we put
\[
Z\left(  s,z\right)  =U\left(  s,z\right)  -V\left(  s,z\right)  =\int
_{0}^{s}[w^{\ast}\left(  \sigma,z\right)  -v^{\ast}\left(  \sigma,z\right)
]d\sigma
\]
by (\ref{U}) and (\ref{V}), one has
\[
L[Z]:=-z^{\beta}\frac{\partial^{2}Z}{\partial z^{2}}-p\left(  s\right)
\frac{\partial^{2}Z}{\partial s^{2}}\leq0
\]
for a.e. $(s,z)\in D:=\left(  0,\left\vert \Omega\right\vert \right)
\times\left(  0,+\infty\right)  $ and the following boundary conditions hold%
\begin{equation}
\begin{array}
[c]{l}%
Z\left(  0,z\right)  =0\quad\forall z\in\left[  0,+\infty\right) \\
\\
\dfrac{\partial Z}{\partial s}\left(  \left\vert \Omega\right\vert
,z\right)  =0\quad\forall z\in\left[  0,+\infty\right) \\
\\
\dfrac{\partial Z}{\partial z}\left(  s,0\right)  \geq0\quad s\in\left(
0,\left\vert \Omega\right\vert \right)  .
\end{array}
\label{princmax}
\end{equation}
In particular%
\begin{equation}
\frac{\partial Z} {\partial\nu}\left(  s,0\right)  =-\frac{\partial Z}
{\partial z}\left(  s,0\right)  \leq0\quad s\in\left(  0,\left\vert
\Omega\right\vert \right)  , \label{der_normale}%
\end{equation}
where $\nu$ is the outward normal to the line segment $(0,|\Omega|)$. We observe that the operator $L$ is elliptic in any point $(s,z)\in D$
hence by Hopf's maximum principle (see \cite{Protter}), $Z$ attains its maximum on the boundary
of $D$, and in the points where the maximum is attained we get
\[
\frac{\partial Z}{\partial\nu}>0.
\]
Hence by \eqref{princmax}, \eqref{der_normale}, this ensures that%
\[
Z \left(  s,z\right)  \leq0\quad s\in\left[  0,\left\vert \Omega\right\vert
\right]
\]
that is%
\[
\int_{0}^{s}w^{\ast}\left(  \sigma,z\right)  d\sigma\leq\int_{0}^{s}v^{\ast
}\left(  \sigma,z\right)  d\sigma\quad s\in\left[  0,\left\vert \Omega
\right\vert \right]
\]
for any $z\in\left[  0,+\infty\right)  .$
\end{proof}

Obviously, since $\phi$ the trace on $\Omega^{\#}$\ of the solution $v$ of
(\ref{eq.4}) and $u$ the trace on $\Omega$\ of the solution $w$ of
(\ref{eq.2}), by Theorem \ref{Th_comparison_w_v} we get Theorem 1.1.

\section{Regularity results}

In this section we are interested in regularity results for solution $u$ of
problem (\ref{eq.1}).\ Using Theorems 1.1 and \ref{Th_comparison_w_v}, we are able to
prove some regularity results of the solution $u$ in terms of the data $f.$
In the following we will use the integral form \eqref{intrep} for the solution $\phi$ to the symmetrized problem \eqref{sym}, namely
\begin{equation}
\phi(x)=-\int_{\Omega^{\#}}\mathcal{G}_{\Omega^{\#}}(x,y)\,f^{\#}(y)\,dy.\label{Green}
\end{equation}
We start by generalizing a well-known result for the classical Laplacian:

\begin{theorem}
\label{THinfinito} Let $u$ the solution to problem \emph{(\ref{eq.1})}, where
$f\in L^{\frac{N}{\alpha},1}(\Omega)$ with $0<\alpha<2$. Then $u\in L^{\infty
}(\Omega)$.
\end{theorem}

\begin{proof}
Let us consider the solution $\phi$ to problem (\ref{sym}). Since $\phi$\ is radially
decreasing, using (\ref{Green}) and (\ref{Est_Green}) we obtain that, for some constant $C$,
\begin{align*}
\Vert\phi\Vert_{L^{\infty}\left(  \Omega^{\#}\right)  }=\phi(0)=  &
\int_{\Omega^{\#}}f^{\#}(y)\,|\mathcal{G}_{\Omega^{\#}}(0,y)|dy\leq
\mathsf{a}\,\mathsf{b}\int_{\Omega^{\#}}\frac{f^{\#}(y)}{|y|^{N-\alpha}}dy\\
&  =\mathsf{a}\,\mathsf{b}\int_{0}^{R_{\Omega}}r^{\alpha-1}f^{\ast}(\omega_{N}r^{N})\,dr=\mathsf{a}\,\mathsf{b}\int
_{0}^{|\Omega|}s^{\frac{\alpha-N}{N}}f^{\ast}(s)\,ds=\mathsf{a}\,\mathsf{b}\Vert f\Vert
_{L^{\frac{N}{\alpha},1}(\Omega)}.
\end{align*}
On the other hand, Theorem 1.1 gives
\[
\Vert u\Vert_{L^{\infty}(\Omega)}\leq\Vert\phi\Vert_{L^{\infty}(\Omega^{\#})}%
\]
and the result follows.
\end{proof}

\begin{remark}\label{rem-THinfinito}
We stress that if $f\in L^{p}(\Omega)$, for some $p>N/\alpha,$ then according
to Lorentz embedding (see Subsection 2.3) by Theorem \emph{\ref{THinfinito}}
we get $u\in L^{\infty}(\Omega).$
\end{remark}
\medskip A consequence of the comparison result of Theorem
\ref{Th_comparison_w_v} is the boundedness of the $\alpha$-extension $w$ of
$u$ in $\mathcal{\overline{C}}_{\Omega}$ when $f\in L^{\frac{N}{\alpha}%
,1}(\Omega)$, for $0<\alpha<2$. To prove this result, we first compute the
solution $v$ to the radial problem (\ref{prob simm}) by using the separation
of variable method. We look for a function $v$, radial with respect to $x$,
such that
\[
v(x,z)=X(|x|)\,W(z).
\]
Putting $v$ inside the first equation of (\ref{prob simm}), we find there must
be a value $\lambda$ such that
\begin{equation}
z^{\beta}\frac{W^{\prime\prime}\left(  z\right)  }{W\left(  z\right)
}=-\Delta_{x}X=\lambda\label{sepmet}%
\end{equation}
that is the function $X(x)=X(|x|)$ solves the classical eigenvalue problem for
the Laplacian%

\begin{equation}
\left\{
\begin{array}
[c]{lll}%
-\Delta_{x}\,X=\lambda X &  & in\text{ }\Omega^{\#}\\
&  & \\
X=0 &  & on\text{ }\partial\Omega^{\#},
\end{array}
\right.  \label{eigenvprob}%
\end{equation}
while $W(z)$ verifies the problem%

\begin{equation}
\left\{
\begin{array}
[c]{l}%
z^{\beta}W^{\prime\prime}\left(  z\right)  -\lambda W\left(  z\right)  =0\\
\\
\lim\limits_{z\rightarrow+\infty}W\left(  z\right)  =0.\\
\end{array}
\right.  \label{eq_Y(z)}%
\end{equation}

Therefore $(\lambda,X)=(\lambda_{k},X_{k})$, for some $k$, where $\left\{\lambda_{k}\right\}
$ and $\left\{X_{k}(|x|)\right\}$ are the eigenvalues and the \emph{radial} eigenfunctions of the Laplace operator in
$\Omega^{\#}$ with zero Dirichlet boundary values on $\partial\Omega^{\#}$,
namely
\begin{equation}
\lambda_{k}=\left(  \frac{\theta_{k}}{R_{\Omega}}\right)  ^{2}\qquad
k=1,2,\ldots\label{eigenv}%
\end{equation}
where $$R_{\Omega}=\left(\frac{|\Omega|}{\omega_{N}}\right)^{1/N}$$ is the radius of the ball $\Omega^{\#}$, $\theta_{k}$ are
the zeros of the Bessel function $J_{(N-2)/2}(z)$ of order \hbox{$(N-2)/2$.}, and
\begin{equation}
X_{k}(r)=\frac{1}{R_{\Omega}\,|J_{\frac{N}{2}}(\theta_{k})|}\left(  \frac
{2}{N\omega_{N}}\right)  ^{1/2}r^{-\frac{N-2}{2}}\,J_{\frac{N-2}{2}}\left(
\frac{\theta_{k}}{R_{\Omega}}\,r\right)  \qquad k=1,2,\ldots. \label{eigenfun}%
\end{equation}
where $r:=|x|$. We recall that the system $\left\{X_{k}(|x|)\right\}$ forms
an orthonormal basis of the space $L_{rad}^{2}(\Omega^{\#})$ made by all
radial functions in $L^{2}$.
\\Then, recalling that the solution $\phi$ of (\ref{sym}) is radially decreasing, we can represent it by
\begin{equation}
\phi(r)=\sum_{k=1}^{\infty}a_{k}X_{k}(r), \label{spectr}%
\end{equation}
where the $a_{k}$ are given by \eqref{ak}, and $c_{k}$ are the Fourier
coefficients of $f^{\#}$ with respect to \eqref{eigenfun}, i.e.
\[
c_{k}=N\omega_{N}\int_{0}^{R_{\Omega}}r^{N-1}\,f^{\ast}(\omega_{N}r^{N}%
)X_{k}(r)dr.
\]
Now, to each eigenvalue $\lambda_{k}$ we associate a solution $W_{k}$ to
problem (\ref{eq_Y(z)}). The equation in (\ref{eq_Y(z)}) is a modified Bessel
equation (see \cite{Polyanin}, \cite{Lebedev}), whose solutions are
combinations of Bessel functions of the third kind. According to the
asymptotic behavior at infinity of the Bessel functions (\cite{Lebedev}), we
have that
\begin{equation}
W_{k}\left(  z\right)  =C_{k}\,H_{k}(z), \label{W(z)}%
\end{equation}
where
\[
H_{k}(z):=\sqrt{z}\,K_{\frac{1}{2-\beta}}\left(  \frac{2}{2-\beta}%
\sqrt{\lambda_{k}}\,z^{\frac{2-\beta}{2}}\right)  ,
\]
$\beta=2\left(  \alpha-1\right)  /\alpha$, the $C_{k}$ are constants and
$K_{\nu}\left(  t\right)  $ is a Bessel function of the third kind. We also notice that
\begin{equation}
H_{k}^{\prime}(z)=-z^{\frac{1-\beta}{2}}\sqrt{\lambda_{k}}\,K_{\frac{1-\beta
}{2-\beta}}\left(  \frac{2}{2-\beta}\sqrt{\lambda_{k}}\,z^{\frac{2-\beta}{2}%
}\right)  . \label{derivH}%
\end{equation}
Finally, using the boundary condition of problem (\ref{prob simm}), we can
write the following explicit expression of $v$ (here $r=|x|$):
\begin{equation}
v\left(  r,z\right)  =\sum_{k=1}^{\infty}X_{k}(r)W_{k}(z)=\frac{1}{R_{\Omega}%
}\left(  \frac{2}{N\omega_{N}}\right)  ^{1/2}r^{-\frac{N-2}{2}}\sum
_{k=1}^{\infty}\frac{C_{k}}{|J_{\frac{N}{2}}(\theta_{k})|}\,J_{\frac{N-2}{2}%
}\left(  \frac{\theta_{k}}{R_{\Omega}}\,r\right)  \,H_{k}(z),
\label{v_Separata_1}%
\end{equation}
with coefficients
\begin{equation}
C_{k}H_{k}^{\prime}(0)=-\frac{(2N\omega_{N})^{1/2}\alpha^{\alpha-1}%
\kappa_{\alpha}}{R_{\Omega}\,|J_{\frac{N}{2}}(\theta_{k})|}\int_{0}%
^{R_{\Omega}}r^{\frac{N}{2}}J_{\frac{N-2}{2}}\left(  \ \theta_{k}\frac
{r}{R_{\Omega}}\right)  \,f^{\ast}(\omega_{N}r^{N})dr. \label{C}%
\end{equation}
Of course the trace  $v\left(  r,0\right)  $\ given in (\ref{v_Separata_1}) coincides
with the solution $\phi$ represented by (\ref{spectr}). Indeed by the
asymptotic behavior (see \cite{Lebedev})
\begin{equation}
K_{\nu}(t)\approx\frac{2^{\nu-1}\Gamma(\nu)}{t^{\nu}}\quad t\rightarrow0,
\label{asymK}%
\end{equation}
then by (\ref{derivH}), (\ref{C}) and \eqref{spectr} we find

\begin{align*}
v\left(  r,0\right)   &  =\sum_{k=1}^{\infty}X_{k}(r)C_{k}\,H_{k}(0)=\\
&  =\frac{\left(  2N\omega_{N}\right)  ^{1/2}}{R_{\Omega}}\sum_{k=1}^{\infty
}\frac{\lambda_{k}^{-\frac{\alpha}{2}}X_{k}(r)}{\left\vert J_{\frac{N}{2}%
}\left(  \theta_{k}\,\right)  \right\vert }\int_{0}^{R_{\Omega}}%
t^{N/2}\,J_{\frac{N-2}{2}}\left(  \frac{\theta_{k}}{R_{\Omega}}t\right)
\,f^{\ast}(\omega_{N}t^{N})dt\\
&  =\phi(x).
\end{align*}

Now we are able to prove the following result.

\begin{theorem}\label{reg. w}
Let $w$ the solution to problem \emph{(\ref{eq.2})}, where $f\in L^{\frac
{N}{\alpha},1}(\Omega)$ with $0<\alpha<2$. Then $w\in L^{\infty}%
(\overline{\mathcal{C}}_{\Omega})$.
\end{theorem}

\begin{proof}
Let $v$ be the solution of (\ref{eq.4}) as in (\ref{v_Separata_1}). By the
asymptotic behavior of the Bessel functions $K_{\nu}$ at infinity (see
\cite{Lebedev}) we deduce that $H_{k}(z)\rightarrow0$ as $z\rightarrow\infty$,
therefore \hbox{$v(r,z)\rightarrow0$} as $z\rightarrow\infty$. Besides, since
$v(x,0)=\phi(x),$ by Theorem \ref{THinfinito} we find that
\hbox{$v\in L^{\infty
}(\overline{\mathcal{C}}_{\Omega^{\#}})$}. Moreover, Theorem
\ref{Th_comparison_w_v} assures that
\begin{equation}
\Vert w(\cdot,z)\Vert_{L^{\infty}(\Omega)}\leq\Vert v(\cdot,z)\Vert
_{L^{\infty}(\Omega^{\#})}\qquad\forall z\in\lbrack0,+\infty),\label{linftext}
\end{equation}
hence $w\in L^{\infty}(\overline{\mathcal{C}}_{\Omega})$.
\end{proof}

We emphasize that the result of Theorem \ref{reg. w} is not new (see for
instance \cite{Colorado}, \cite{Cab3}), although our techniques make us able
to achieve the sharper $L^{\infty}$ estimate \eqref{linftext}.\\

Now we provide \emph{new} regularity results when $f$ belongs to Lorentz spaces
$L(p,r)$ for \hbox{$p<N/\alpha$}, obtaining the
generalization of the corresponding classical regularity result for the Laplacian.

\begin{theorem}
Let $u$ be the solution to problem \emph{(\ref{eq.1})}, where $f\in
L^{p,r}(\Omega)$ with $$\frac{2N}{N+\alpha}\leq p<\frac{N}{\alpha}$$ and
$r\geq1$. Then $u\in L^{q,r}(\Omega)$ with $$q:=\frac{Np}{N-\alpha p}.$$
\end{theorem}

\begin{proof}
Extending $f$ to zero outside $\Omega^{\#}$, inserting inequality \eqref{Est_Green} into \eqref{Green} we find
\[
|\phi(x)|\leq\mathsf{a}\,\mathsf{b}\,(f^{\#}\ast |x|^{\alpha-N}).
\]
Then applying Theorem 2.3 with the choices $g=|x|^{\alpha-N}$, $p_1=p$, $p_2=N/(N-\alpha)$, $q_1=r$, $q_2=\infty$\color{red}, \color{black} we have
$s=q=Np/(N-\alpha p)$, $t=r$ and
\begin{align}
\left\Vert \phi\right\Vert _{L^{q,r}
\left(  \Omega^{\#}\right)  }&\leq\mathsf{a}\,\mathsf{b}\left\Vert f^{\#}\ast |x|^{\alpha-N}\right\Vert _{L^{q,r}
\left(  \Omega^{\#}\right)  }\nonumber\\& \leq 3q\,\mathsf{a}\,\mathsf{b}\, ||f||_{L^{p,r}(\Omega^{\#})}\,|||x|^{\alpha-N}||_{L^{N/(N-\alpha),\infty}(\R^{N})}\nonumber\\&
=3q\,\mathsf{a}\,\mathsf{b}\, ||f||_{L^{p,r}(\Omega^{\#})}\label{conv}.
\end{align}
Finally by Theorem 1.1 we get
\[
\Vert u\Vert_{L^{q,r}(\Omega)}\leq\left\Vert \phi\right\Vert _{L^{q,r}%
\left(  \Omega^{\#}\right)  }%
\]
and inequality \eqref{conv} allows to conclude.
\end{proof}

\section{Best constant in $L^{\infty}$ estimate}

In virtue of Theorem \ref{THinfinito}, if $f\in L^{N/\alpha,1}(\Omega)$ there is a constant
$\mathsf{C}$ such that
\begin{equation}
\Vert u\Vert_{L^{\infty}}\leq\mathsf{C}\Vert f\Vert_{L^{N/\alpha,1}(\Omega)}.
\label{limitat}%
\end{equation}
Due to the form of the Green function in \eqref{GreenB(0,R)}, it seems quite
difficult to face the problem of finding the \emph{best} value for
$\mathsf{C}$ in \eqref{limitat}. Nevertheless, we remark that this becomes
reasonably easy when one replaces the $L^{N/\alpha,1}$- norm of $f$ at the
right-hand side of \eqref{limitat} with an $L^{p}$ norm, for some $p>N/\alpha$ (which is possible, by Remark \ref{rem-THinfinito}). In fact, since
the solution $\phi$ is \emph{radially decreasing}, in order to get an
$L^{\infty}$ estimate of $\phi$ it is enough to look for a sharp upper bound
of $\phi(0)$. To this end, we first observe that \eqref{Green} yields
\begin{equation}
\phi(0)=-\int_{0}^{|\Omega|}\psi((s/\omega_{N})^{1/N})f^{\ast}(s)ds,
\label{supremum}%
\end{equation}
where
\begin{equation}
\psi(t):=-2^{-\alpha}\frac{\Gamma\left(  \frac{N}{2}\right)  }{R_{\Omega}%
^{2}\pi^{N/2}}\Gamma\left(  \frac{\alpha}{2}\right)  ^{-2}\,t^{\alpha
-N}\,{\displaystyle\int_{0}^{\frac{R_{\Omega}^{2}}{t^{2}}(R_{\Omega}^{2}%
-t^{2})}}\frac{s^{\frac{\alpha}{2}-1}}{\left(  \frac{s}{R_{\Omega}^{2}%
}+1\right)  ^{N/2}}\,ds. \label{Psi}%
\end{equation}
We remark that it is possible to write an explicit form of the integral at the right-hand side of \eqref{Psi}. Indeed, we know that
(\emph{e.g.} see \cite{Table})
\[
\int_{0}^{w}\frac{s^{\frac{\alpha}{2}-1}}{(\frac{s}{R_{\Omega}^{2}}+1)^{N/2}%
}\,ds=2\frac{R_{\Omega}^{2\alpha}\,w^{\alpha/2}}{\alpha}%
\,_{2}F_{1}\left(  \frac{N}{2};\frac{\alpha}{2};1+\frac{\alpha}{2};-R_{\Omega
}^{2}w\right)
\]
where $_{2}F_{1}(\cdot\,;\cdot\,;\cdot\,;\,z)$ denotes the \emph{Gauss
hypergeometric function}. Therefore by \eqref{Psi}
\[
\psi((s/\omega_{N})^{1/N})=\mathcal{B}_{N,\alpha}\,s^{\frac{\alpha-N}{N}%
}\left[  \frac{1}{s^{\alpha/N}}\left(  |\Omega|^{2/N}-s^{2/N}\right)
^{\alpha/2}\,_{2}F_{1}\left(  \frac{N}{2};\frac{\alpha}{2};1+\frac{\alpha}%
{2};\frac{|\Omega|^{4/N}}{\omega_{N}^{4/N}s^{2/N}}\left(  s^{2/N}%
-|\Omega|^{2/N}\right)  \right)  \right]
\]
where
\[
\mathcal{B}_{N,\alpha}:=-\frac{2^{1-\alpha}\,\Gamma\left(  \frac{N}{2}\right)
}{\alpha\Gamma\left(  \frac{\alpha}{2}\right)  ^{2}\pi^{N/2}}\left(
\frac{|\Omega|}{\omega_{N}}\right)  ^{\frac{3\alpha-2}{N}}\omega
_{N}^{1-(\alpha/N)}.
\]
So if we set
\[
\varphi(s):=\frac{1}{s^{\alpha/N}}\left(  |\Omega|^{2/N}-s^{2/N}\right)
^{\alpha/2}\,_{2}F_{1}\left(  \frac{N}{2};\frac{\alpha}{2};1+\frac{\alpha}%
{2};\frac{|\Omega|^{4/N}}{\omega_{N}^{4/N}s^{2/N}}\left(  s^{2/N}%
-|\Omega|^{2/N}\right)  \right).
\]
using Hölder inequality in \eqref{supremum} we have
\[
|\phi(0)|\leq|\mathcal{B}_{N,\alpha}|\,\Vert f\Vert_{L^{p}(\Omega)}\left(
\int_{0}^{|\Omega|}s^{\frac{\alpha-N}{N}p^{\prime}}\left[  \varphi(s)\right]
^{p^{\prime}}ds\right)  ^{1/p^{\prime}}.
\]
We point out that the function $\varphi$ is bounded in $[0,|\Omega|]$
(see also the picture below), so the integral at the right-hand side of the
last inequality converges if and only if $p>N/\alpha$.

\begin{center}
\includegraphics{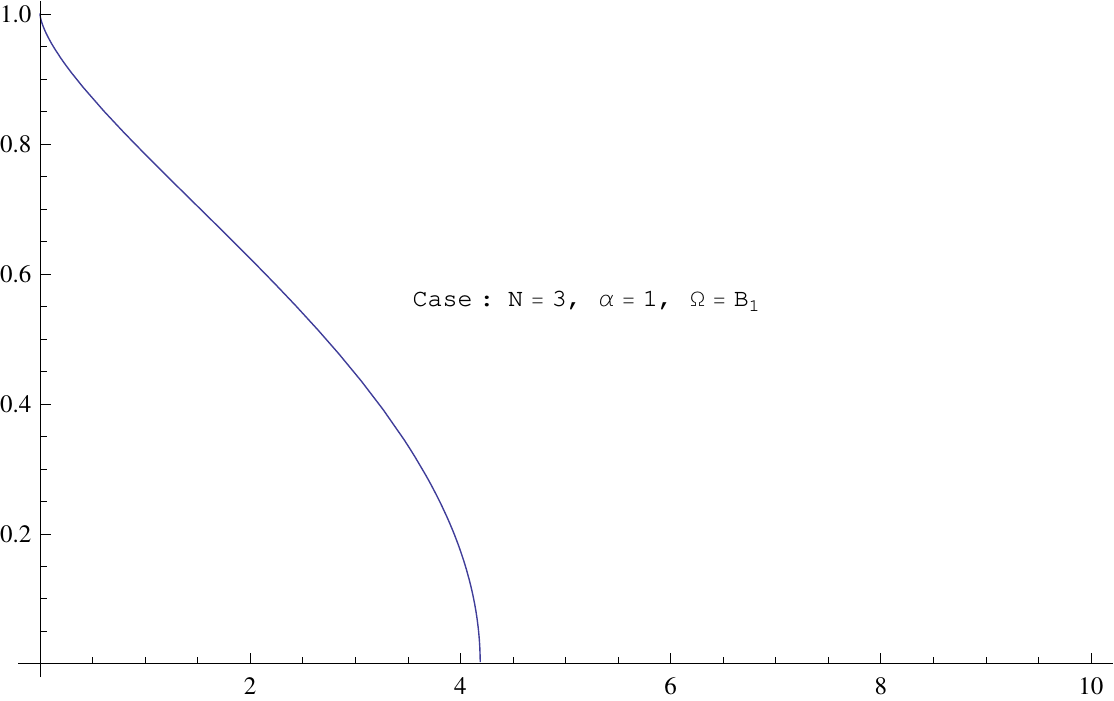}
\end{center}
The best constant $\mathsf{C}( N,p,\alpha,\Omega)$ in (\ref{limitat})
is then
\[
\mathsf{C}(N,p,\alpha,\Omega):=|\mathcal{\ B}_{N,\alpha}|\,\left(  \int
_{0}^{|\Omega|}s^{\frac{\alpha-N}{N}p^{\prime}}\varphi(s)^{p^{\prime}}ds
\right)  ^{1/p^{\prime}}.
\]

\begin{example}
Let us calculate the best constant $\mathsf{C}$ in the case of the square root
of the Laplacian $\sqrt{-\Delta}$ (i.e. the case $\alpha=1$), when $N=3$ and
$\Omega=B(0,1)$. In this case, we have the following, explicit form of the
Gauss hypergeometric function:
\[
_{2}F_{1}\left(  \frac{3}{2};\frac{1}{2};\frac{3}{2},-z\right)  =\frac
{1}{\sqrt{z+1}}.
\]
Then
\[
\varphi(s)=\left(  \frac{3}{4\pi}\right)  ^{1/3}\sqrt{\left(  \frac{4}{3}%
\pi\right)  ^{2/3}-s^{2/3}}%
\]
and we have, by a change of variable,
\begin{align*}
\mathsf{C}(p)  &  =\frac{(2\pi)^{1/p^{\prime}}}{2\pi^{2}}\left(  \int_{0}%
^{1}t^{\frac{1}{2}-p^{\prime}}(1-t)^{\frac{p^{\prime}}{2}}dt\right)
^{1/p^{\prime}}\\
&  =\frac{(2\pi)^{1/p^{\prime}}}{2\pi^{2}}\,B\left(  \frac{p-3}{2(p-1)}%
,\frac{3p-2}{2(p-1)}\right)  ^{(p-1)/p},
\end{align*}
where $B(\cdot,\cdot)$ is the Euler beta function.
\end{example}


\end{document}